\documentclass[preprint, noinfoline,addressatend,imslayout]{imsart}
\RequirePackage[colorlinks,citecolor=blue,urlcolor=blue]{hyperref}
\usepackage[authoryear]{natbib}

\usepackage{amsthm,amsmath,amsfonts,amssymb}
\usepackage{enumerate}

\usepackage{bm,accents}
\usepackage{graphicx}
\usepackage{breakurl}
\usepackage{color}

\startlocaldefs

\setattribute{tablecaption}{shape}{}

\setattribute{tablename} {skip}{.~}
\newtheorem{thm}{Theorem}
\newtheorem{cor}[thm]{Corollary}
\newtheorem{lemma}[thm]{Lemma}

\newcommand{\norm}[1]{\left\Vert#1\right\Vert}

\newcommand{\indi}[1]{1_{\left\{#1\right\}}}
\newcommand{\sumn}{\sum_{i=1}^{n}}

\newcommand{\R}{\mathbb{R}}
\newcommand{\eps}{\varepsilon}

\newcommand{\argmax}{\operatornamewithlimits{argmax}}

\newcommand{\ep}[0]{\mathbb{P}}
\newcommand{\CK}{\mathcal{K}}
\newcommand{\CB}{\mathcal{B}}

\bibpunct{(}{)}{;}{a}{,}{,}
\def\spacingset#1{\renewcommand{\baselinestretch}%
	{#1}\small\normalsize} \spacingset{1.6}

\begin{document}
	
	\begin{frontmatter}
		

		\title{Semiparametric Analysis of the Proportional Likelihood Ratio Model and Omnibus Estimation Procedure\protect\thanksref{T1}}
		\runtitle{Semiparametric Analysis}
		\thankstext{T1}{This research was supported by Grant No.~2016126 from the United States-Israel Binational Science Foundation (BSF). We would like to thank Anastasios Tsiatis and Eric~J. Tchetgen Tchetgen for very helpful discussions.}
		
		\begin{aug}
			\author{\fnms{Yair} \snm{Goldberg}\ead[label=e1]{yairgo@technion.ac.il}}
			and \author{\fnms{Malka}
				\snm{Gorfine}\ead[label=e2]{gorfinem@tauex.tau.ac.il}}

			\affiliation{Technion and Tel-Aviv University}
			
			\address{Yair Goldberg\\The Faculty of Industrial Engineering and Management\\Technion,  Haifa 3200003, Israel\\ \printead{e1}}

			\address{Malka Gorfine\\Department of Statistics and Operation Research\\ Tel-Aviv University\\
				Ramt Aviv, 69981, Israel\\ \printead{e2}}
			
			\runauthor{Goldberg and Gorfine}
		\end{aug}
		\begin{abstract}
			We provide a semi-parametric analysis for the proportional likelihood ratio model, proposed by Luo \& Tsai (2012). We study the tangent spaces for both the parameter of interest and the nuisance parameter, and obtain an explicit expression for the efficient score function. We propose a family of Z-estimators based on the score functions, including an approximated efficient estimator. Using inverse probability weighting, the proposed estimators can also be applied to different missing-data mechanisms, such as right censored data and non-random sampling. A simulation study that illustrates the finite-sample performance of the estimators is presented.
		\end{abstract}

	\end{frontmatter}

\section{Introduction}\label{sec:intro}
Recently, \citet{luo_proportional_2012} proposed a semi-parametric  proportional likelihood ratio model that extends generalized linear models. The  model assumes that the joint distribution of the response $Y$ and the $q\times 1$ covariate vector $X$ is
\begin{align}\label{eq:model1}
p_{Y,X}(y,x)= p_{Y\mid X} (y \mid x)p_{X}(x) =\frac{ \exp(\beta^T xy) g(y)}{\int  \exp(\beta^T xy) dG(y)}\eta(x)\,,
\end{align}
where $\beta\in\R^{q}$ is the Euclidean parameter of interest and $G(y)$ and $\eta(x)$ are the nuisance parameters. Here $G(y)\equiv P(Y\leq y\mid X=0)$ is a baseline distribution function with density function $g(y)$ with respect to some dominating measure $\nu$; and $\eta(x)\equiv p_{X}(x)$ is the density of $X$ with respect to some dominating measure. A comprehensive discussion of the model interpretation can be found in \citet{luo_proportional_2012}, \cite{chan_nuisance_2013}, and references therein. Semi-parametric maximum likelihood estimators of $\beta$ and $G$ were given by~\citet{luo_proportional_2012} and the convergence of their iterative estimation algorithm was proved by~\citet{davidov_convergence_2013}. 

\cite{chan_nuisance_2013} showed that under a certain missing-data model (which will be discussed later), or when $Y$ is subject to doubly-random truncation, the above model is invariant with respect to $\beta$, but not necessarily with respect to $G$. By cleverly using the pairwise pseudo-likelihood of \citet{liang2000regression}, \cite{chan_nuisance_2013} presented a pseudo-score equation for $\beta$ that is free of the functional parameter $G$, and such that $\beta$ is consistently estimated. This estimator is computationally efficient, but not statistically efficient. Estimation procedures based on right-censored data and on longitudinal data were proposed by  \cite{zhu_likelihood_2014} and \cite{luo_moment-type_2015}, respectively. 

The proportional likelihood ration model~\eqref{eq:model1} is a special case of the prospective likelihood discussed by~\citet{chen2003note} in the context of outcome-dependent samples. \citet[][Eq.~5]{chen2003note} considered the model 
\begin{align*}
p_{Y\mid X} (y \mid x) =\frac{\rho(y,y_0,x,x_0,\beta) g(y)}{\int  \rho(y,y_0,x,x_0,\beta) dG(y)}\,,
\end{align*}
where 
\begin{align}\label{eq:generalized_odds_ratio}
\rho(y,y_0,x,x_0,\beta)=\frac{p_{Y\mid X}(y\mid x)}{p_{Y\mid X}(y\mid x_0)} \left\{\frac{p_{Y\mid X}(y_0\mid x)}{p_{Y\mid X}(y_0\mid x_0)}\right\}^{-1}\,,
\end{align}
and $(y_0, x_0)$  is a sample point. In this model, $g(y)=p_{Y\mid X,R}(y|x_0,1)$ and $R$ is the sampling indicator, taking value $1$ if it is included in the
sample and $0$ otherwise. The function $\rho$ is referred to as the generalized odds ratio \citep{liang2000regression}. For generalized linear models with canonical link function, and taking $(y_0,x_0)\equiv(0,0)$, $\rho$ has the from $\exp\{\beta^Txy\}$. \citet{chen2007semiparametric} and \citet{tchetgen2010doubly} further extended~\eqref{eq:generalized_odds_ratio} by allowing conditioning on an additional covariate vector $Z$. They study the nuisance tangent space for this model. They also considered three parametric models,  for $\rho $, $p_{Y\mid X,Z}(y\mid x_0,z)$, and $p_{X \mid  Y,Z}(x \mid y_0,z)$. They then proposed doubly robust estimators which are consistent when the model for $\rho$ is correct and either $p_{Y\mid X,Z}(y\mid x_0,z)$ or $p_{X \mid  Y,Z}(x \mid y_0,z)$ is correctly specified.

The contribution of this work is twofold. First, it provides a comprehensive semi-parametric analysis of the \citeauthor{luo_proportional_2012}'s proportional likelihood model \eqref{eq:model1} above. The proofs involve projections in Hilbert spaces and solving integral equations. Second, the semi-parametric theory we develop in this work yields an omnibus estimation procedure including problems previously studied separately, such as missing data, doubly-truncated or censored data. Under the missing-data setting, the proposed estimation approach is not limited to the specific missing-data model of \cite{chan_nuisance_2013}. Moreover, in contrast to \citet{chen2007semiparametric}, \citet{tchetgen2010doubly}, and  \cite{chan_nuisance_2013}, where only  $\beta$ is estimated, we nonparametricly estimate $G$ as well, which makes our approach useful also for prediction. The utility of our novel estimation procedure is demonstrated via extensive simulation study. Efficient implementation of the proposed estimation procedure, as well as of that of \citet{luo_proportional_2012} and \cite{chan_nuisance_2013}, is implemented in the R package PLR (which can be freely downloaded from \url{https://github.com/yairgoldy/PLR}).

\section{Semi-parametric Analysis}\label{sec:theory}
The proportional likelihood ratio semi-parametric model can be written as the set of densities
\begin{align*}
\mathcal{P}=&\LARGE\{p_{Y,X}\{y,x;\beta,g(\cdot),\eta(\cdot)\}\Large\}=\LARGE\{p_{Y\mid  X}\{y\mid  x;\beta,g(\cdot)\}\eta(x) \LARGE\}
\end{align*}
where $g(y)$ and $ \eta(x)$ are the nuisance parameters. The respective true values of the parameters are denoted by $\beta_0$, $g_0(\cdot)$, and $\eta_0(\cdot)$. In the sequel, we assume the standard smoothness and regularity conditions \citep[see, for example,][Definition~A1]{newey1990semiparametric}. Let $\mathcal{H}$ denote the tangent space for $\mathcal{P}$, where $\mathcal{H}$ is the Hilbert space of all $q$-dimensional random functions $a(Y,X)$ that satisfy $E\{a(Y,X)\}=0$ and have finite variance, equipped with the inner product $\langle a_1,a_2\rangle =E(a_1^T a_2)$. Let $\Lambda\subset \mathcal{H}$ be the nuisance tangent space with respect to the parameters $g_0(\cdot)$ and $\eta_0(\cdot)$ \citep[see][Chapter~4, Defintion 1]{Tsiatis2006Semiparametric}. The marginal density of $Y$ is 
\begin{align*}
p_Y(y)=g(y)\int\frac{ \exp(\beta^T  xy) \eta(x)}{\int  \exp(\beta^T  xy) dG(y)}dx\, ,
\end{align*}
where $g(y)$ is the density of $Y$ given $ X=0$.

We start by calculating the nuisance parameters tangent space, motivated by Theorem 4.2 of \cite{Tsiatis2006Semiparametric} which states that the influence function of any asymptotically linear and regular (RAL) estimator is orthogonal to the nuisance tangent space. We then show how to calculate the projection of any score function on the nuisance tangent space. As a result, we are able to provide an explicit representation of the efficient score,  which is the projection of the score function with respect to $\beta$ on the orthogonal complement of the nuisance tangent space \citep[see][Definition 4.2]{Tsiatis2006Semiparametric}. While in practice the projections are difficult to compute since they  are an infinite sequences of alternating expectations, based on an approximately-projected scores, we provide a novel family of estimators.

\begin{lemma}\label{lem:lambda_12}
	Let $\Lambda_1$ and $\Lambda_2$ be the nuisance tangent spaces with respect to $g(\cdot)$ and $\eta(\cdot)$, respectively. 
	Then,
	\begin{align*}
	&\text{(i)}\quad  \Lambda_1 = \left\{h(Y)-E\{h(Y)\mid  X\}\,:\, h(Y)\text{ is a $q$-dimensional vector-valued function}\right\}\\
	&\text{(ii)}\quad  \Lambda_2 = \left\{\alpha(X) \,:\, \alpha(X)\text{ is a $q$-dimensional vector-valued function such that } E\{\alpha(X)\}=0\right\}
	\end{align*}
\end{lemma}
The proof of Lemma \ref{lem:lambda_12} is provided in Appendix 1.1. The nuisance parameters $\eta(\cdot)$ and $g(\cdot)$ are variationally independent, that is, any choice of $\eta$ and $g$ results in a density in the model $\mathcal{P}$  \citep[see definition at][page 53]{Tsiatis2006Semiparametric}. Moreover, we have the following result.
\begin{lemma}
	The spaces $\Lambda_1$ and $\Lambda_2$ are orthogonal.
\end{lemma}
\begin{proof}
	For every $h(Y)-E\{h(Y)\mid  X\}\in\Lambda_1$ and $\alpha(X)\in \Lambda_2$,
	\begin{align*}
	E\left([h(Y)-E\{h(Y)\mid  X\}]\alpha(X)\right)=&E\{E\left[(h(Y)-E\{h(Y)\mid  X]\}\alpha(X)\mid  X\right]\}\\
	=&E\left([E\{h(Y)\mid  X\}-E\{h(Y)\mid  X\}]\alpha(X)\right)=0,
	\end{align*}
	as needed.
\end{proof}

By Theorem 5.2 of \citet{Tsiatis2006Semiparametric}, the projection of a function $a(Y,X)$ on $\Lambda$, denoted by $\Pi(a\mid\Lambda)$, can be written as $\Pi(a\mid \Lambda)=\Pi(a\mid \Lambda_1)+\Pi(a\mid \Lambda_2)$. In the following, the projection of $a$ on each nuisance tangent space is computed. Let 
\begin{align*}
\CK \{a(Y,X)\} = E \{ a(Y,X) \mid Y\} - 
E [  E \{ a(Y,X) \mid X \} \mid Y ] \, .
\end{align*}
Define the linear operator $\CB$ by $
\CB\{f(Y) \} = E[E \{ f(Y)|X \}|Y]$,
and let $I:\mathcal{H}\mapsto \mathcal{H}$  denotes the identity mapping. 
\begin{thm}\label{thm:projections}
	The respective projections of $a(Y,X)$ on $\Lambda_1$ and $\Lambda_2$ are
	\begin{align*}
	&(i)\quad    \Pi\{a(Y,X)\mid \Lambda_1\}=
	\left(I-\CB \right)^{-1} \CK\{a(Y,X)\}- E\left[\left(I-\CB \right)^{-1} \CK\{a(Y,X)\}\mid X\right]\, , 
	\\
	&(ii)\quad\Pi\{a(Y,X)\mid \Lambda_2\}= E\{a(Y,X)\mid X\}\,. 
	\end{align*}
\end{thm}
See the detailed proof in the Appendix. The following statement is a direct consequence of Theorem~\ref{thm:projections}, and  the cornerstone  for generating the proposed omnibus estimation procedure.

\begin{cor}\label{cor:proj_on_lambda_bot}
	Let $\Lambda^\bot$ be the orthogonal complement of the nuisance tangent space $\Lambda$ in $\mathcal{H}$. Then
	\begin{align*}
	\Pi \{ a(Y,X)\mid \Lambda^\bot \} =  &\, a(Y,X)-E\{  a(Y,X)\mid X\}  
	\\
	&\qquad- \sum_{j=0}^\infty \left(\CB^j \CK\{a(Y,X)\}- E\left[\CB^j \CK\{a(Y,X)\}\mid X\right]\right)
	\,.
	\end{align*}
\end{cor}
Similar result was also obtained by \citet[][see the Appendix there]{chen2007semiparametric} who used the prospective and retrospective likelihoods to shows that the space $\Lambda$ can be written as an intersection between two linear subspaces of $\mathcal{H}$. He then used von-Neumann projection theorem \citep[][Theorem A.4.1]{bickel1993efficient} to calculate the projection. This is different from the above theorem, as
here we explicitly compute the nuisance tangent spaces with respect to $g(\cdot)$ and $\eta(\cdot)$.

Let $S_\beta(Y,X)$ be the score function for $\beta$, namely, the derivative of $\log p_{Y,X}\{y,x;\beta,g_0(\cdot),\eta_0(\cdot)\}$ with respect to $\beta$ evaluated at the true parameter value $\beta_0$. 
The efficient score function, $S_{\mathrm{eff}}(Y,X)$, is defined as the projection of $S_\beta(Y,X)$ on $\Lambda^\perp$. 
\begin{lemma}\label{lem:efficient_score}
	The efficient score $S_{\mathrm{eff}}(Y,X)$ equals
	$$XY- E(XY\mid X) - \sum_{j=0}^\infty \left[\CB^j \CK(XY)-E\{\CB^j \CK(XY)\mid X\}\right]\,. $$
\end{lemma}  
See proof in the Appendix. Since the projection of $a(Y,X)$ on $\Lambda_1$ is an infinite series where the norm of subsequent terms decrease, we approximate the projection onto the nuisance tangent space by using only the first few terms. Our proposed approximated efficient score  is defined by 
\begin{align*}
XY- E(XY\mid X) - E(XY \mid Y) + E \{ E(XY \mid Y) \mid X \} \,.
\end{align*}
The asymptotic and finite-sample properties are studied in the following sections.

\section{Estimation}
In this section we propose a family of estimators for the parameter of interest $\beta_0$ using the theory developed in Section~\ref{sec:theory}. An estimator $\widehat{\beta}$ for  $\beta_0$ is called asymptotically linear if there exists a $q$-dimensional random vector $\varphi(Y,X)$, such that $E\{\varphi(Y,X)\}=0^{q\times 1}$, and $n^{1/2}(\widehat{\beta}-\beta_0)=n^{-1/2}\sum_{i=1}^n \varphi(Y_i,X_i)+o_p(1)$ such that $E\{ \varphi(Y,X)\varphi(Y,X)^T\}$ is finite and nonsingular \citep[][Chapter~3]{Tsiatis2006Semiparametric}. By Theorem~4.2 of \cite{Tsiatis2006Semiparametric}, if $\varphi$ is an influence function for an estimator~$\widehat{\beta}$, then~$\varphi$ is orthogonal to the nuisance
tangent space~$\Lambda$, or  more formally,
\begin{align*}
\Pi\{\varphi(Y,X)\mid \Lambda\}\equiv 0^{q\times 1}\,.
\end{align*}
By Theorem 4.3 of \citet{Tsiatis2006Semiparametric}, every  regular asymptotically linear (RAL)  estimator for $\beta_0$ has a unique influence function. Therefore, we propose a family of Z-estimators, based on their influence functions using the fact that the influence functions must lie in $\Lambda^{\bot}$. 

Fix any function $a(Y,X)$, $a:\R\times \R^q\mapsto \R^q$; then $a(Y,X)-E\{a(Y,X)\}\in\mathcal{H}$.
Let 
\begin{align*}
m_{0,a}(y,x) &\equiv a(y,x)-E\{  a(Y,X)\mid X=x\}  
- \CK\{a(y,x)\} \\
&=a(x,y)- E\{a(X,Y)\mid X=x\} 
- E\{a(X,Y) \mid Y=y\} \\
&\quad+ E [E
\{a(X,Y)  \mid Y\} \mid X=x] 
\,.
\end{align*}  
The function $m_{0,a}(y,x)$ is an approximated projection of $a$ on $\Lambda^\bot$. Let
\begin{align}\label{eq:S_U_Vcontinuous}
\begin{split}
S_{a}(x,\beta, G)&=\int a(s,x) \exp(\beta^T xs)dG(s)\,,
\\
U_{a}(y,\beta,G)&=\int\frac{ a(y,x) \exp(\beta^T  xy) \eta(x)}{S_{1}(x,\beta,G)}dx\,,
\\
V_{a}(x,\beta,G)&=\frac{1}{S_{1}(x,\beta,G)}\int \frac{U_{a}(y,\beta,G)}{U_{1}(y,\beta,G)} \exp(\beta^T xy)dG(y)\,,
\end{split}
\end{align}
where $S_1$ and $U_1$ are the functions $S_{a}$ and $U_{a}$, respectively, for $a\equiv 1$.
Using the definition of the density $g(\cdot)$ and some algebraic manipulations, it can be shown that 
\begin{align*}
&E\{a(Y,X)\mid X=x\} =\frac{S_{a}(x,\beta_0,G_0)}{S_{1}(x,\beta_0,G_0)}\,,
\\
&E\{a(X,Y) \mid Y=y\} =\frac{U_{a}(y,\beta_0,G_0)}{U_{ 1}(y,\beta_0,G)}\,,
\end{align*}
and
\begin{align*}
&E [E
\{a(X,Y)  \mid Y\} \mid X=x] = V_{a}(x,\beta_0 ,G_0)\,.
\end{align*}
Let
\begin{align}\label{eq:m_with_suv}
m_{a}(y,x,\beta,G)= a(y,x)-\frac{S_{a}(x,\beta,G)}{S_{1}(x,\beta,G)} -\frac{U_{a}(y,\beta,G)}{U_{1}(y,\beta,G)} +V_{a}(x,\beta,G)\,,
\end{align}
and note that, by the above discussion, $m_{0,a}(y,x)=m_{a}(y,x,\beta_0,G_0)$. The function $m_{a}(y,x,\beta,G) $ is used for defining the estimating equations.

Suppose that we observe independent and identically distributed random pairs $(Y_1,X_1),\ldots,(Y_n,X_n)$  from the distribution function $p_{Y,X}\{y,x;\beta_0,g_0(\cdot),\eta_0(\cdot)\}$. Let $Y_{(1)}, \ldots, Y_{(K)}$ be the ordered distinct observed values of $Y$. For a fixed value of $\beta$, by Theorem~2 of \cite{luo_proportional_2012}, the profile likelihood maximizer for $G$, denote by $\widehat{G}_{\beta}$, has jumps $\widehat{p}(\beta)=\{\widehat{p}_1(\beta),\ldots,\widehat{p}_K(\beta)\}^T$ at $Y_{(1)}, \ldots, Y_{(K)}$. For a vector of probabilities $p\in\R^{K}$, write
\begin{align}\label{eq:S_U_V}
\begin{split}
\widehat{S}_{a}(x,\beta,p)&=\sum_{k=1}^K a(Y_{(k)},x) \exp(\beta^T x Y_{(k)})p_k\,,
\\
\widehat{U}_{a}(y,\beta,p)&=\frac{1}{n}\sum_{i=1}^n  \frac{a(y,X_i)\exp(\beta^T X_i y)}{\widehat{S}_{1}(X_i,\beta,p)}\,,
\\
\widehat{V}_{a}(x,\beta,p)&=\frac{1}{\widehat{S}_{1}(x,\beta,p)}\sum_{k=1}^{K} \frac{\widehat{U}_{a}(Y_{(k)},\beta,p)}{\widehat{U}_{1}(Y_{(k)},\beta,p)} \exp(\beta^T xY_{(k)})p_k\,.
\end{split}
\end{align}
Then, for every function $a(y,x)$, we propose the following estimating equation for $\beta$:
\begin{align}\label{eq:estimation_eq}
\frac{1}{n}\sum_{i=1}^n m_{n,a}\{Y_i,X_i,\beta,\widehat{p}(\beta)\}=0 \,,
\end{align}
where $m_{n,a}$ is defined similarly to $m_{a}$ in~\eqref{eq:m_with_suv} by replacing $S_{a}$, $U_{a}$ and $V_{a}$ with $\widehat{S}_{a}$, $\widehat{U}_{a}$, and $\widehat{V}_{a}$, respectively. An estimator of $G$ is then obtained by taking $\widehat G=\widehat G_{\widehat \beta}$. Note that, by Lemma~\ref{lem:efficient_score}, choosing $a(Y,X)=XY$ in $m_{a}$ yields estimating equations which are based on the approximated efficient score.

\section{Asymptotic Results -- the discrete setting}

Consider a discrete random variable $Y$ with finite support, and assume
\begin{enumerate}
	\renewcommand{\labelenumi}{(A\arabic{enumi})}
	
	\item $X$ takes values in a compact set $\mathcal{X}\subset\R^q$, and $\beta_0$ is an interior point of a bounded set $\mathcal{B}\subset \R^{q}$.\label{as:boundedXB}
	
	\item The function $E[m_{a}\{Y,X,\beta,p^*(\beta)\}]$ has a unique zero at $\beta_0$, and its derivative with respect to $\beta$ is invertible at $\beta_0$, where $p^*(\beta)$ is the limit of $\widehat{p}(\beta)$.\label{as:unique_z}
	
\end{enumerate}
Let $p_0\in\R^K$ be the vector of true probabilities of $(Y_{(1)},\ldots,Y_{(K)})$ given $X=0$.

\begin{thm}\label{thm:estimation}
	Under the assumptions above,  $\widehat{\beta}$ and $\widehat{p}(\widehat{\beta})$ are consistent estimators for $\beta_0$ and  $p_0$, respectively. Moreover, both $n^{1/2}(\widehat{\beta}-\beta_0)$ and  $n^{1/2}\{\widehat{p}(\widehat{\beta})-p_0\}$ converge to mean-zero Gaussian vectors.
\end{thm}

The asymptotic variance of $\widehat{\beta}$ can be estimated empirically by standard estimating-equations tools. However, the computation is rather complex. Instead, a bootstrap approach is recommended, which is justified by \citet[][Theorem 13.4]{Kosorok2008}. We conjecture that Theorem 1 holds for general distributions of $Y$.  Proving this is challenging. A typical first step is to show that uniformly in $\beta$, the profile likelihood maximizer $\widehat G_\beta$ converges to some limit $G_{\beta,0}$. In other words, one needs to show that uniformly in $\beta$, the random process $\argmax \frac{1}{n}\sum_{i=1}^{n}\text{pl} (Y_i,X_i,\beta,G)$ converges to a fixed limit, where $\text{pl}$ is the profile likelihood, and the maximization is taken over all step distribution functions with jumps at the sample points. However, the maximizer of the profile log-likelihood is given only implicitly as a solution of a nonlinear set of equations \citep[Theorem~2]{luo_proportional_2012}. Since no explicit solution is given for the maximizer, standard empirical process techniques are difficult to employ. This is different from proofs such as those in  \citet{murphy1997maximum} and \cite{luo_proportional_2012}, that use nonparametric maximum likelihood, since their proof requires convergence  only at a the value of the true parameters $(\beta_0,G_0)$.  This is also different from the locally semi-parametric proofs of \citet{chen2007semiparametric}  and \citet{tchetgen2010doubly} as a parametric model for $g$ is assumed.
\begin{proof}[Proof of Theorem~\ref{thm:estimation}]
	
	\textbf{Part 1: Convergence of $\widehat{p}(\beta)$ to a limit $p^*(\beta)$.} As explained in the proof of Theorem~2 of \citet{luo_proportional_2012}, for each fixed $\beta$, $\widehat{p}(\beta)$ is obtained by maximization the likelihood (2.2) of \citet{vardi1985empirical}. Note that for $\beta\neq\beta_0$, this maximization is carried out with respect to a misspecified model. Indeed, since $K$ is fixed, the log likelihood $l$ of one observation $(y,x)$ for a fixed $\beta$ is 
	\begin{align*}
	l(y,x;p,\beta)=\sum_{k=1}^K\indi{y=Y_{(k)}}\left\{\beta^T xY_{(k)}+\log(p_k)\right\} \sum_{k=1}^Kp_k\exp\{\beta^T xY_{(k)}\}\,.
	\end{align*}
	By Theorem~3.2 of \citet{white1982maximum}, for every fixed $\beta$, $n^{1/2}
	\{\widehat{p}(\beta)-p^*(\beta)\}$ converges to a mean-zero Gaussian vector. Note that $l(y,x;\beta,p)$ and its first and second derivatives are all continuous function of $\beta$, and as a result of Theorem~3.2 of \citet{white1982maximum}, the limit $p^*(\beta)$ is also continuous in~$\beta$.
	
	\textbf{Part 2: Consistency.} 
	The outline of the consistency proof is as follows. First, we define the $\widehat{\beta}$ as a zero of an estimating equation $\Psi_n(\beta)=0$. We show that $\Psi_n(\beta)$ converges uniformly to a function $\Psi(\beta)$ which has a unique zero at $\beta_0$ and has the property that if $\{\beta^{(m)}\}_{m=1}^{\infty}$ is any sequence for which $\Psi\left(\beta^{(m)}\right)\rightarrow 0$, then $\beta^{(m)}\rightarrow \beta_0$. By Theorem~2.10 of \citet{Kosorok2008}, this proves consistency of $\widehat{\beta}$ to $\beta_0$.

	By~\eqref{eq:estimation_eq}, the estimating equation is 
	\begin{align*}
	\Psi_n(\beta)\equiv\frac1n\sum_{i=1}^{n}\left[a(Y_i,X_i)-\frac{\widehat{S}_{a}\{X_i,\beta,\widehat{p}(\beta)\}}{\widehat{S}_{1}\{X_i,\beta,\widehat{p}(\beta)\}} -\frac{\widehat{U}_{a}\{Y_i,\beta,\widehat{p}(\beta)\}}{\widehat{U}_{1}\{Y_i,\beta,\widehat{p}(\beta)\}} +\widehat{V}_{a}\{X_i,\beta,\widehat{p}(\beta)\}\right]=0\,.
	\end{align*}
	Let $\Psi(\beta)\equiv E\left[m_{a}(Y,X,\beta,p^*(\beta))\right]$. By Assumption~(A\ref{as:unique_z}), $\Psi(\beta)$ has a unique zero at $\beta_0$. Since $p^*(\beta)$ is continuous and $E\{m_{a}(Y,X,\beta,p)\}$ is also continuous in both $\beta$ and $p$, so is $\Psi(\beta)$ as a composition of continuous functions. Hence, for any sequence $\{\beta^{(m)}\}$, if $\Psi(\beta^{(m)})\rightarrow 0$, then $\beta^{(m)}\rightarrow \beta_0$.
	
	We now prove that $\Psi_n(\beta)$ converges uniformly to $\Psi(\beta)$. Define
	\begin{align*}
	\mathcal{S}_a\equiv&\left\{S_{a}(X,\beta,p)\,:\, \beta\in\mathcal{B}, p\in\mathcal{P}\right\},\quad a=0,1\,,
	\\
	\mathcal{S}_{a}/\mathcal{S}_{1}\equiv&\left\{\frac{S_{a}(X,\beta,p)}{S_{1}(X,\beta,p)}\,:\, \beta\in\mathcal{B}, p\in\mathcal{P}\right\}\,,
	\end{align*}
	where $S_{1}$ and $S_{a}$ are defined in~\eqref{eq:S_U_V}, and $\mathcal{P}\equiv\left\{p\in\R^K: p_k\geq 0,\,\, \sum_{k=1}^K p_k=1\right\}$. By Corollary~9.32 of \citet{Kosorok2008}, the classes $\mathcal{S}_1$, $\mathcal{S}_a$, and $\mathcal{S}_{a}/\mathcal{S}_{1}$, are Donsker since by Assumption~(A\ref{as:boundedXB}), $\beta$, $p$ and $x$ are bounded, the exponent function is Lipschitz on compact sets, and the function $a$ is bounded. Hence
	\begin{align}\label{eq:DonskerS}
	\begin{split}
	\sup_{\beta\in\mathcal{B},p\in\mathcal{P}}\norm{\ep_n S_{a}(X,\beta,p)-E\left\{S_{a}(X,\beta,p)\right\}}&\rightarrow 0\,,
	\\
	\sup_{\beta\in\mathcal{B},p\in\mathcal{P}}\norm{\ep_n \frac{S_{a}(X,\beta,p)}{S_{1}(,\beta,p)}-E\left\{\frac{S_{a}(X,\beta,p)}{S_{1}(X,\beta,p)}\right\}}&\rightarrow 0\,,
	\end{split}
	\end{align}
	where $\ep_n$ is the empirical measure such that for every function $f$, $\ep_n f(Y,X)\equiv\sum_{i=1}^n f(Y_i,X_i)$.
	By Assumption~(A\ref{as:boundedXB}), $S_{1}$ is uniformly bounded from below by a positive constant. Using the same argument as above, $ \mathcal{U}_a\equiv\large\{a(y,x)\exp(\beta^T x y)/S_{1}(x,\beta,p)\large\}$ is Donsker. Hence,  
	\begin{align*}
	&\sup_{\substack{\beta\in\mathcal{B},p\in\mathcal{P} \\ y\in\{Y_{(1)},\ldots,Y_{(k)} \} }}\norm{\widehat{U}_{a}(y,\beta,p)-U_{a}(y,\beta,p)}
	\\
	&=\sup_{\substack{\beta\in\mathcal{B},p\in\mathcal{P} \\ y\in\{Y_{(1)},\ldots,
			Y_{(k)} \} }}\left\|\ep_n \frac{a(y,X)\exp(\beta^T X y)}{S_{1}(X,\beta,p)}  
	- E\left[\frac{a(y,X)\exp(\beta^T X y)}{S_{1}(X,\beta,p)}\right]\right\|\rightarrow 0\,.
	\end{align*}
	Applying Corollary~9.32(iv) of \citet{Kosorok2008} to the classes $\mathcal{U}_1$ and $\mathcal{U}_a$ yields that the quotient  $\mathcal{U}_a/\mathcal{U}_1$ is also Donsker. Hence, one can show that
	\begin{align}\label{eq:DonskerU}
	\begin{split}
	&\sup_{\beta\in\mathcal{B},p\in\mathcal{P} }
	\norm{\ep_n\frac{\widehat{U}_{a}(Y_i,\beta,p)}{\widehat{U}_{1}(Y_i,\beta,p)}-E\left\{\frac{U_{a}(Y,\beta,p)}{U_{1}(Y,\beta,p)}\right\}}\rightarrow 0\,.
	\end{split}
	\end{align}

	
	Similar arguments shows that $\mathcal{V}_a\equiv\{V_a(y,\beta,p): \beta\in\mathcal{B},p\in\mathcal{P} \}$
	is also Donsker and that
	\begin{align}\label{eq:DonskerV}
	\sup_{\beta\in\mathcal{B}}\|\ep_n\widehat{V}_{a}(y,\beta,p)-V_{a}(y,\beta,p)\| =o_p(1)\,.
	\end{align}
	Consequently, by the definitions of $m_{n,a}$ and $m_{a}$, and by Eqs~\eqref{eq:DonskerS},~\eqref{eq:DonskerU}, and~\eqref{eq:DonskerV},
	\begin{align}\label{eq:Psi_n-Psi}
	\begin{split}
	\sup_{\beta\in\mathcal{B}}\norm{\Psi_n(\beta)-\Psi(\beta)}
	\leq&\sup_{\beta\in\mathcal{B}}\|\ep_n
	m_{n,a}(Y,X,\beta,p_n(\beta))-E[m_{a}\{Y,X,\beta,p_n(\beta)\}]\|
	\\
	& +
	\sup_{\beta\in\mathcal{B}}\norm{E[m_{a}\{Y,X,\beta,p_n(\beta)\}]-E[m_{a}\{Y,X,\beta,p^*(\beta)\}]}
	\\
	\leq&\sup_{\beta\in\mathcal{B},p\in\mathcal{P}}\norm{\ep_n
		m_{n,a}(Y,X,\beta,p)-E\{m_{a}(Y,X,\beta,p)\}}
	\\
	& +
	\sup_{\beta\in\mathcal{B}}\norm{E[m_{a}\{(Y,X,\beta,p_n(\beta)\}]-E[m_{a}\{Y,X,\beta,p^*(\beta)\}]}
	=o_p(1)\,,
	\end{split}
	\end{align}
	which concludes the consistency proof.

	\textbf{Part 3: Normality.} For any function $f(y,x,\beta,p)$ define 
	$
	D_\beta f(y,x,\beta,p)\equiv \frac{\partial }{\partial\beta}f(y,x,\beta,p)
	$
	and define $D_p f(y,x,\beta,p)$ similarly. Define
	\begin{align*}
	D_{n} m_{n,a}\{y,x,\beta,\widehat{p}(\beta)\}&\equiv\frac{\partial }{\partial\beta}m_{n,a}\{y,x,\beta,\widehat{p}(\beta)\}\\
	&=D_\beta m_{n,a}\{y,x,\beta,\widehat{p}(\beta)\}+D_p m_{n,a}\{Y,X,\beta,\widehat{p}(\beta)\}\frac{\partial \widehat{p}(\beta)}{\partial\beta}\,,
	\end{align*}
	and define $D_0 m_{a}(y,x,\beta,p^*(\beta))$ similarly. 
	By using similar arguments to~\eqref{eq:Psi_n-Psi}, we get 
	\begin{align}\label{eq:convegence_div}
	\begin{split}
	&\sup_{\beta\in\mathcal{B},p\in\mathcal{P}}
	\norm{ \ep_n
		D_p m_{n,a}(Y,X,\beta,p)-E\{D_p m_{a}(Y,X,\beta,p)\}}=o_p(1)\,,
	\\
	&\quad\sup_{\beta\in\mathcal{B}}
	\;\;\norm{ \ep_n
		D_{n}m_{n,a}\{Y,X,\beta,\widehat{p}(\beta)\}-E[D_0 m_{a}\{Y,X,\beta,p^*(\beta)\}]}=o_p(1)\,.
	\end{split}
	\end{align}
	
	We have
	\begin{align*}
	0= \Psi_n(\widehat{\beta})=& \,\ep_n m_{n,a}\left\{Y,X,\widehat{\beta},\widehat{p}(\widehat{\beta})\right\}\\
	=&\,\ep_nm_{n,a}\left(Y,X,\beta_0,p_0\right)+\ep_nD_{p}m_{n,a}\left(Y,X,\beta_0,p_0\right)\left\{\widehat{p}(\beta_0)-p_0\right\}\\
	&+\ep_n D_nm_{n,a}\left(Y,X,\beta_0,p_0\right)(\widehat{\beta}-\beta_0)+o_p\left\{ (\widehat{\beta}-\beta_0)+n^{-1/2}\right\}
	\end{align*}
	since 
	$\widehat{p}(\beta_0)-p_0=O_p(n^{-1/2})$ by Part 1 above. Write
	\begin{align*}
	&\ep_n\left\{m_{n,a}(Y,X,\beta,p)-m_{a}(Y,X,\beta,p)\right\}
	\\
	&\quad=\ep_n\left\{\frac{U_{a}(Y,\beta,p)}{U_{1}(Y,\beta,p)} -\frac{\widehat{U}_{a}(Y,\beta,p)}{\widehat{U}_{1}(Y,\beta,p)}\right\}
	\\
	&\quad\quad+\ep_n\left[\frac{1}{S_{1}(X,\beta,p)}\sum_{k=1}^{K} \left\{\frac{\widehat{U}_{a}(Y_{(k)},\beta,p)}{\widehat{U}_{1}(Y_{(k)},\beta,p)}-\frac{U_{a}(Y_{(k)},\beta,p)}{U_{1}(Y_{(k)},\beta,p)}\right\} \exp\{\beta^T XY_{(k)}\}p_k\right]\,.
	\end{align*}
	Then
	\begin{align*}
	&\ep_n\left\{\frac{U_{a}(Y,\beta,p)}{U_{1}(Y,\beta,p)} -\frac{\widehat{U}_{a}(Y,\beta,p)}{\widehat{U}_{1}(Y,\beta,p)}\right\}
	\\
	&=\ep_n\left\{\frac{U_{a}(Y,\beta,p)\widehat{U}_{1}(Y,\beta,p)-\widehat{U}_{a}(Y,\beta,p)U_{1}(Y,\beta,p)}{U_{1}^2(Y,\beta,p)}\right\}+o_p(1)
	\\
	&=\frac1{n^2}\sumn\sum_{j=1}^{n}\frac{U_a(Y_i,\beta,p)u_1(Y_i,X_j,\beta,p)-U_1(Y_i,\beta,p)u_a(Y_i,X_j,\beta,p)}{U_1^2(Y_i,\beta,p)}+o_p(1)
	\end{align*}
	where
	\begin{align*}
	u_a(y,x,\beta,p)=\frac{a(y,x)\exp(\beta^T x y)}{S_{1}(x,\beta,p)}\,,
	\end{align*}
	and thus behaves like a V-statistic up to an $o_p(1)$ term. Using similar arguments, one can show that
	\begin{align}\label{eq:m_n-m}
	n^{1/2}\ep_n\left\{m_{n,a}(Y,X,\beta,p)-m_{a}(Y,X,\beta,p)\right\}
	\end{align} 
	converges to a mean-zero Gaussian vector. Hence, by~\eqref{eq:convegence_div}, 
	\begin{align*}
	0&= n^{1/2}\ep_nm_{a}\left\{Y,X,\beta_0,p_0\right\}
	\\
	&\quad+n^{1/2}\ep_n\left\{m_{n,a}(Y,X,\beta,p)-m_{a}(Y,X,\beta,p)\right\}
	\\
	&\quad+E\left\{ D_p m_{a}\left(Y,X,\beta_0,p_0\right)\right\}n^{1/2}\left\{\widehat{p}(\beta_0)-p_0\right\}
	\\
	&\quad+E\left[ D_0 m_{a}\left\{Y,X,\beta_0,\widehat{p}(\beta_0)\right\}\right]n^{1/2}(\widehat{\beta}-\beta_0)+o_p\left\{n^{1/2} (\widehat{\beta}-\beta_0)+1\right\}\,.
	\end{align*}
	Multiplying both sides of this equation by $-E\left[ D_0 m_{a}\left\{Y,X,\beta_0,\widehat{p}(\beta_0)\right\}\right]^{-1}$, and using the Donsker property for $\mathcal{S}_a/\mathcal{S}_1$, $\mathcal{U}_a/\mathcal{U}_1$ and $\mathcal{V}_a$, the fact that $\widehat{p}(\beta_0)-p_0$ converges to a mean-zero Gaussian vector, and~\eqref{eq:m_n-m}, we obtain that $n^{1/2} (\widehat{\beta}-\beta_0)$ converges to a Gaussian random vector. Finally,
	\begin{align*}
	n^{1/2}\left\{\widehat{p}(\widehat{\beta})-p_0\right\}=& n^{1/2}\left\{\widehat{p}(\widehat{\beta})-\widehat{p}(\beta_0)\right\}+n^{1/2}\left\{\widehat{p}(\beta_0)-p_0\right\}\\
	&=
	D_p \{p^*(\beta_0)\}n^{1/2}(\widehat{\beta}-\beta_0)+n^{1/2}\left\{\widehat{p}(\beta_0)-p_0\right\}+o_p(1)
	\,,
	\end{align*}
	which converges to a mean-zero Gaussian by Part~1 and the argument above.\end{proof}

\section{Incomplete and Sampling-Biased Data}
So far we assumed that the data $(Y_1,X_1),\ldots,(Y_n,X_n)$ are fully observed and identically distributed. In the literature, the proportional likelihood model~\eqref{eq:model1} with incomplete data was considered in a case by case manner. For example, \citet{chan_nuisance_2013} shows how to handle missing data when the probability of missingness has a specific form, namely, 
\begin{align}\label{eq:chan_decomposition}
\mbox{pr}\left(R=1|Y=y,X=x\right)= h_1(y)h_2(x)\,,
\end{align}
where $R$ is the indicator for non-missing data, and $h_1$ and $h_2$ are arbitrary functions.  He also considers the double-truncation setting when the truncation is independent of both $Y$ and $X$. \citet{zhu_likelihood_2014} discusses the right-censored setting when the censoring variable $C$ is independent of the pair $(Y,X)$. Other settings, such as selection-biased data where the randomization is not proper, were not studied so far.  

The estimating equation~\eqref{eq:estimation_eq} enables us to provide an omnibus solution for all the problems that discussed aboves. Indeed, when the selection probabilities are known or can be estimated, and similarly, when the censoring or truncation probabilities can be estimated, one can use the inverse weighing methods \citep{Robins94}. In the missing data and censoring settings, let $R_i$ be an indicator equals one for complete observations.  Let $W_i=\mbox{pr}(R_i=1 \mid Y_i, X_i)$, and let $\widehat{W}_i$ be a consistent estimator of $W_i$. For the sampling-biased setting, let $W_i$ be the sampling probability of observation $i$ and $R_i\equiv 1$. For a fixed $\beta$, let $\widehat{p}_W(\beta)$ be the weighted-profile-likelihood estimator obtained as the maximizer of
\begin{align*}
\sum_{i=1}^n\frac{R_i}{\widehat{W}_i}\left\{\indi{Y_i=Y_{(k)}}\left(\beta^T X_iY_{(k)} +\log p_k\right)-\log\sum_{k=1}^K
p_k\exp\left(\beta^T X_iY_{(k)}\right)\right\}\,.\end{align*}

Then, the solution $\widehat{\beta}_{W}$ of the estimating equation
\begin{align}\label{eq:estimation_eq_we}
\frac{1}{n}\sum_{i=1}^n \frac{R_i}{\widehat{W}_i}m_{n,a}\{Y_i,X_i,\beta,\widehat{p}_{W}(\beta)\}=0^{q\times 1}\,,
\end{align}
is a consistent estimator of $\beta_0$. Moreover, $\widehat{p}_W(\widehat{\beta}_W)$ is a consistent estimator of $G$. The finite-sample performance of this estimator for the missing-data setting is demonstrated in Section \ref{sec:simulation}.

\section{Simulation Study}\label{sec:simulation}

We compare our method to two existing methods: the MLE of \citet{luo_proportional_2012} and the pseudo-likelihood method of \citet{chan_nuisance_2013}.  The two scenarios of \citet{luo_proportional_2012} were considered. Specifically, the covariate vector consists of $X=(X_1,X_2)^T$, where $X_2$ follows a zero-mean normal distribution with standard deviation 0.5, and given $X_2$, $X_1$ follows the Bernoulli distribution with success probability $\exp(1-X_2)/\{1+\exp(1-X_2)\}$. The value of the true parameters are $\beta=(\beta_1,\beta_2)^T=(-1,-1)^T$. In Setting 1, $Y$ is continuous and the baseline density is defined by 
$$
g(y) = \{\Phi(0.5)\}^{-1} (0.5 \pi)^{-1/2} \exp\{-2(y-0.25)^2\} \;\;\;\;\;\; y \geq 0  \, , 
$$
where $\Phi$ is the standard normal cumulative distribution function. In Setting 2, $Y$ is discrete with 
$$ 
g(y) = (1+y)3^y \exp(-3)/\{4y! \} \;\;\;\;\;\; y=0,1,2,\ldots \, . 
$$ 
Each setting consists of 1000 replicates and sample sizes 100, 200, 400 and 800. We compare the bias in estimating $\beta_1$, $\beta_2$, and the distance 
$\int |\widehat{G}(t)-G(t)|dt$. The simulation results are summarized in Table~\ref{table_a} and Figure~\ref{fig_12}, in the Appendix.  The proposed estimator coincides with the maximum likelihood estimator of \citet{luo_proportional_2012}, and behaves similarly to that of \citet{chan_nuisance_2013}.  

Two additional settings, with missing covariates, are considered. In both settings, the full data were generated  as in Setting~1. The probability of observing complete data is
$$\mbox{pr}(R=1|Y,X_1,X_2)={\exp(1-X_2)}\{1+\exp(1-X_2)\}^{-1} \, ,$$ in Setting~3, and $$\mbox{pr}(R=1|Y,X_1,X_2)={\exp(1-X_2-2Y)}\{1+\exp(1-X_2-2Y)\}^{-1}$$
in Setting~4. Note that the missing probability in Setting~3 follows~\eqref{eq:chan_decomposition} and hence can be consistently estimated by ignoring the missing observations. This is no longer true for Setting~4. The simulation results are summarized in Table \ref{table_b} and Figure~\ref{fig_34}. While all three methods work similarly in Setting~3, only the proposed method succeeds in estimating $\beta$ consistently. Moreover, the bias of the proposed method in estimating~$G$ converges to zero, while for the other two methods the bias converges to a constant.


\begin{table}[]
	\caption{Simulation results with $\beta_0=(-1,-1)^T$: empirical mean (empirical standard deviation)\label{table_b}}
	{\begin{tabular}{llcccc}
			&	\multicolumn{5}{c}{Setting 3}  \\
			&            & 100          & 200          & 400          & 800          \\
			$\beta_1$    & Lou \& Tsai & -1.09 (0.94) & -0.99 (0.65) & -1.00 (0.43)    & -1.00 (0.29)    \\
			& Chan        & -1.10 (0.97)  & -1.00 (0.67)    & -1.01 (0.44) & -1.01 (0.30)  \\
			& Proposed    & -1.07 (0.95) & -0.98 (0.66) & -1.00 (0.44)    & -0.99 (0.30)  \\
			$\beta_2$    & Lou \& Tsai & -1.06 (0.90)  & -1.05 (0.61) & -1.02 (0.41) & -1.02 (0.27) \\
			& Chan        & -1.08 (0.92) & -1.06 (0.63) & -1.02 (0.42) & -1.02 (0.27) \\
			& Proposed    & -1.06 (0.90)  & -1.05 (0.62) & -1.02 (0.42) & -1.02 (0.27) \\
			Distance & Lou \& Tsai & 0.08 (0.04)  & 0.06 (0.03) & 0.06 (0.02) & 0.03 (0.01)  \\
			& Chan        & 0.08 (0.04)  & 0.06 (0.03)  & 0.04 (0.02)  & 0.03 (0.01)  \\
			& Proposed    & 0.08 (0.04)  & 0.06 (0.03)  & 0.04 (0.02)  & 0.03 (0.01)  \\
			&	\multicolumn{5}{c}{Setting 4}  \\
			$\beta_1$    & Lou \& Tsai & -1.06 (1.42) & -1.01 (0.93) & -1.00 (0.62)    & -1.01 (0.43) \\
			& Chan        & -1.08 (1.47) & -1.03 (0.96) & -1.00 (0.64)    & -1.02 (0.44) \\
			& Proposed    & -1.32 (4.32) & -1.21 (1.20)  & -1.09 (0.78) & -1.05 (0.53) \\
			$\beta_2$    & Lou \& Tsai & -1.61 (1.30)  & -1.56 (0.86) & -1.49 (0.58) & -1.47 (0.38) \\
			& Chan        & -1.65 (1.35) & -1.58 (0.89) & -1.50 (0.59)  & -1.48 (0.39) \\
			& Proposed    & -1.31 (1.33) & -1.16 (0.93) & -1.01 (0.61) & -0.94 (0.41) \\
			Distance & Lou \& Tsai & 0.13 (0.06)  & 0.12 (0.05)  & 0.11 (0.04)  & 0.11 (0.03)  \\
			& Chan        & 0.16 (0.06)  & 0.12 (0.05)  & 0.11 (0.04)  & 0.11 (0.03)  \\
			& Proposed    & 0.12 (0.01)   & 0.09 (0.07)  & 0.07 (0.05)  & 0.05 (0.03) 	\end{tabular}}
\end{table}

\begin{figure}[!t]
	\begin{center}
		\includegraphics[width=1\linewidth]{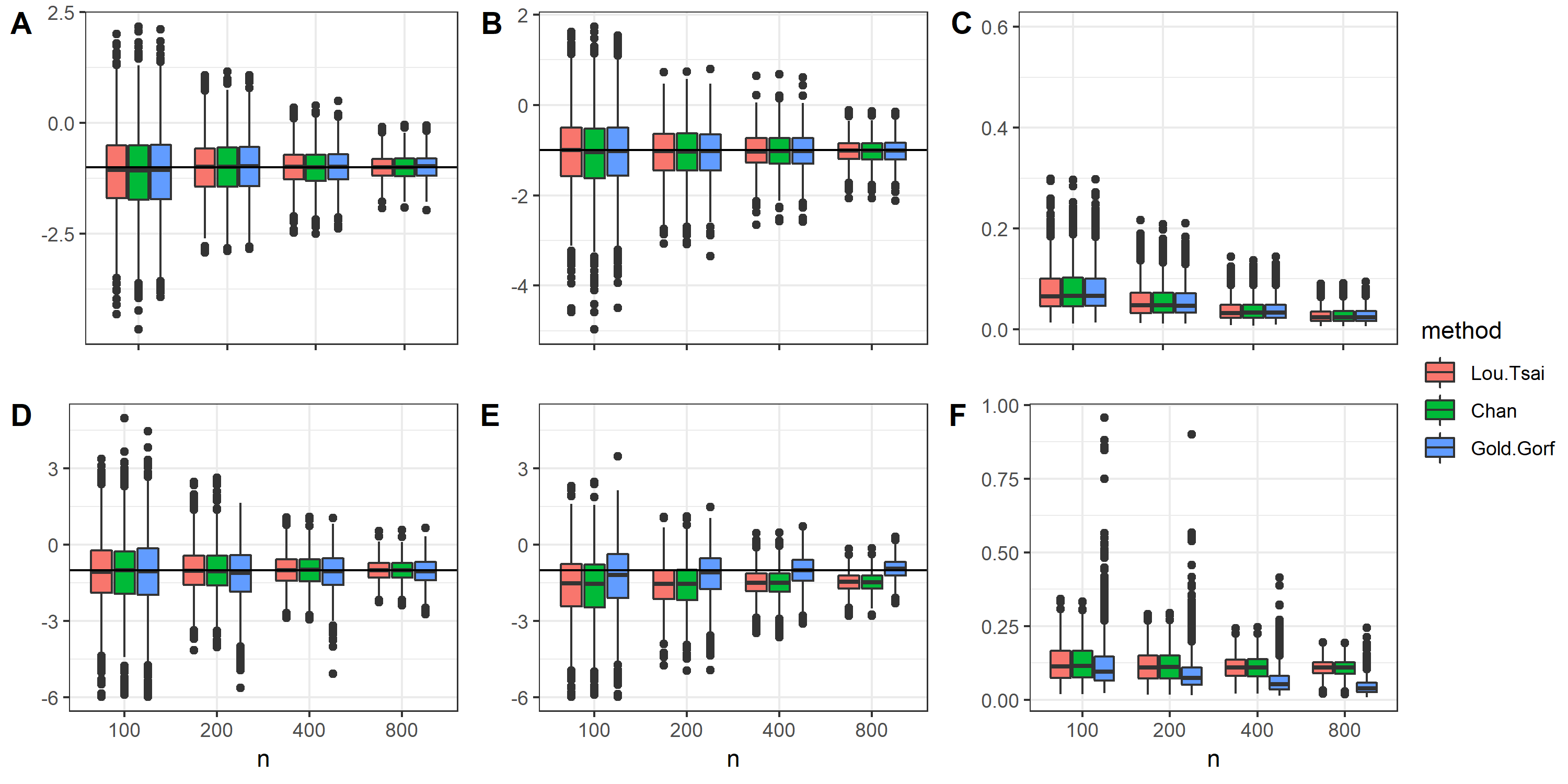}
		\caption{Results of Settings~3 and~4 are summarized in the top and the bottom lines, respectively. A~and~D present the bias related to $\beta_1$, B~and~E present that of $\beta_2$, and~C and~F  present the bias related to $G$. 
			\label{fig_34}}
	\end{center}
\end{figure}

\section*{Acknowledgement}
This research was supported by Grant No.~2016126 from the United States-Israel Binational Science Foundation (BSF). We would like to thank Anastasios Tsiatis and Eric J. Tchetgen Tchetgen for very helpful discussions.

\appendix
\section*{Appendix}\label{sec:proofs}
\begin{proof}[Proof of Lemma~\ref{lem:lambda_12}]
	Assertion~(ii) follows from ~\citet{Tsiatis2006Semiparametric}, Theorem~4.6. For assertion~(i), consider the parametric submodel
	\begin{align*}
	p_{y\mid  x}(y\mid  x;\beta_0,\gamma)=  \frac{ \exp(\beta_0^T  xy) g_0(y)\exp\{h(y)^T \gamma\}}{\int  \exp(\beta_0^T  xy) g_0(y)\exp\{h(y)^T \gamma\}dy}\,,
	\end{align*}
	where $\gamma\in\R^q$ is the nuisance parameter and $h(Y)$ is $q$-dimensional vector-valued bounded function. Clearly, the true model is obtained for $\gamma=0$. Moreover, $g_0(y)\exp\{h(y)^T \gamma\}$ is indeed a density of~$Y$. The score function with respect to this submodel is given by
	\begin{align*}
	S_\gamma(y,x)=&\left.\frac{\partial}{\partial \gamma}\log
	p_{y\mid  x}(y\mid  x;\beta_0,\gamma)\right|_{\gamma=0}
	\\
	=&\left.\frac{\partial}{\partial \gamma}\left[h(y)^T \gamma-\log\int  \exp(\beta_0^T  xy) g_0(y)\exp\{h(y)^T \gamma\}dy \right]\right|_{\gamma=0}
	\\
	=&\,h(y)-E\{h(Y)\mid  X=x\}\,.
	\end{align*}
	We have demonstrated that any element in $\Lambda_1$ defined above is an element of a parametric submodel nuisance tangent space. Therefore, to complete the proof we need to show that the linear space spanned by the score vector with respect to $\gamma$ for any parametric submodel is contained in $\Lambda_1$. The log-density with respect to a parametric submodel can be written as
	\begin{align*}
	\beta_0^T  xy+\log g(y;\gamma)-\log\int  \exp(\beta_0^T  xy) g(y;\gamma)dy +\log \eta_0(x)\,.
	\end{align*}
	Taking the derivative with respect to the parametric submodel $\gamma$ and substituting the true value of the parameter, denoted by $\gamma_0$, we obtain
	\begin{align*}
	S_{\gamma}(y,x)=&\,\frac{\frac{\partial}{\partial \gamma}g(y;\gamma_0)}{g(y;\gamma_0)}-\frac{\int  \exp(\beta_0^T  xy)\frac{\frac{\partial}{\partial \gamma} g(y;\gamma_0)}{g(y;\gamma_0)}g(y;\gamma_0)dy}{\int  \exp(\beta_0^T  xy) g(y;\gamma_0)dy}
	\\
	=&\,\frac{\frac{\partial}{\partial \gamma}g(y;\gamma_0)}{g(y;\gamma_0)}-E\left\{\left.\frac{\frac{\partial}{\partial \gamma}g(Y;\gamma_0)}{g(Y;\gamma_0)}\right|  X=x\right\}\,.
	\end{align*}
	Multiplying the score $S_{\gamma}$ by a conformable matrix yields an element of $\Lambda_1$, which concludes the proof.
\end{proof}

\begin{proof}[Proof of Theorem~\ref{thm:projections}]
	The second assertion follows from Lemma~4.3 of \citet{Tsiatis2006Semiparametric}. For the first assertion, since $\Lambda_1\bot\Lambda_2$, it is enough to first project $a(Y,X)$ on $\Lambda_2^{\bot}$ and then on $\Lambda_1\subseteq \Lambda_2^{\bot}$. By  assertion~(ii), $\Pi\{a(Y,X)\mid \Lambda_2^\bot\}=a(Y,X)-E\{a(Y,X)\mid X\}$. Thus it is enough to find the projection of functions of the form $ a(Y,X)-E\{ a(Y,X)\mid X\}$ on $\Lambda_1$, for functions $a$ such that $E\{ a(Y,X)\}=0$. Since all functions in $\Lambda_1$ are of the form $h(Y)-E\{h(Y)\mid X\}$ for some $h(Y)$, we would like to find a function $h^*(Y)$ such that
	\begin{align*}
	E\left\{\left( [a(Y,X)-E\{a(Y,X)\mid X\}]-[h^*(Y)-E\{h^*(Y)\mid X\}]\right)^T [h(Y)-E\{h(Y)\mid X\}]\right\}=0
	\end{align*}
	for all $h(Y)$. Since 
	\begin{align*}
	&E\left\{\left( [a(Y,X)-E\{a(Y,X)\mid X\}]-[h^*(Y)-E\{h^*(Y)\mid X\}]\right)^T E\{h(Y)\mid X\}\right\} \\
	&\quad = E\left[ E\left\{\left( [a(Y,X)-E\{a(Y,X)\mid X\}]-[h^*(Y)-E\{h^*(Y)\mid X\}]\right)^T E\{h(Y)\mid X\} \mid X \right\}   \right] = 0 \,
	\end{align*}
	it is enough to find $h^*(Y)$ such that 
	\begin{align*}
	E\left\{\left( [a(Y,X)-E\{a(Y,X)\mid X\}]-[h^*(Y)-E\{h^*(Y)\mid X\}]\right)^T h(Y) \right\}=0
	\end{align*}
	for all $h(Y)$ in $\Lambda_1$. This implies that
	\begin{align*}
	E\left\{\left( [a(Y,X)-E\{a(Y,X)\mid X\}]-[h^*(Y)-E\{h^*(Y)\mid X\}]\right) \mid Y \right\}=0.
	\end{align*}
	Equivalently, we would like to find $h^*(Y)$ that solves the integral equations
	\begin{align}\label{eq:hstar}
	(I-\CB)\{ h^*(Y)\} = \CK\{a(Y,X) \} \, , 
	\end{align}
	where the operators $I$, $\CK$ and $\CB$ are defined in Section~\ref{sec:theory}. 
	
	We now show that $\CB$ is a contraction operator,  that is $\left\|\CB \{a(Y,X)\} \right\| \leq 1-\epsilon$ for some $\epsilon > 0$ for all functions $a(Y,X)$ such that $\|a(Y,X)\|=1$. Denote by 
	\begin{align*}
	\Lambda_Y= \left\{h(Y):\, h(Y)\text{ is a $q$-dimensional vector-valued function}\right\}\,,\\
	\Lambda_X = \left\{\alpha(X)\,:\, \alpha(X)\text{ is a $q$-dimensional vector-valued function}\right\}\,.
	\end{align*}
	By~\citet[][Theorem~4.6]{Tsiatis2006Semiparametric}, $\CB\{a(Y,X)\}=\Pi\{\Pi(a(Y,X)|\Lambda_X)|\Lambda_Y\}$. For any $a(Y,X)$ such that $E[a(Y,X)]=0$, and $\|a(Y,X)\|=1$, by the Pythagorean theorem,
	\begin{align*}
	\|\CB\{a(Y,X)\}\|=\|\Pi\{\Pi(a(Y,X)|\Lambda_X)|\Lambda_Y\}\|\leq \|a(Y,X)\|=1\,.
	\end{align*} 
	Assume that  there is no positive $\eps$ such that $\left\|\CB \{a(Y,X)\} \right\| \leq (1-\epsilon) \left\| a(Y,X) \right\|$. Then, there is a sequence $\{a_n(Y,X)\}_{n=1}^\infty$ such that $\|a_n(Y,X)\|=1$ and $\lim_{n\rightarrow\infty} \|a_n(Y,X)\|=1 $. By Alaoglu's theorem \citep{weidmann2012linear}, every bounded sequence contains a weakly convergent subsequence. Let $a_0$ be a limit of such a subsequence. Since $\CB$ is a projection operator, $\|\CB\{a_0(Y,X)\}\|=\|a_0(Y,X)\|$ which implies that $a_0(Y,X)$ is a function only of $Y$. Note that 
	\begin{align*}
	1 =\|\CB\{a_0(Y,X)\}\|\equiv \|\Pi\{\Pi(a(Y,X)|\Lambda_X)|\Lambda_Y\}\|\leq \|\Pi(a(Y,X)|\Lambda_X)\| \leq \|a(Y,X)\|=1\,.
	\end{align*} 
	Hence $\|\Pi(a(Y,X)|\Lambda_X)\|=1$ which implies that $a_0(Y,X)$ is a function only of $X$. Since the only function that can satisfies both conditions is a constant function, and since this function needs to have zero expectation, we arrive at a contradiction.
	
	By \citet[Lemma 10.5]{Tsiatis2006Semiparametric}, since $\CB$ is a contraction operator, $
	h^*(Y) = \sum_{j=0}^{\infty}\CB^j \CK\{a(Y,X)\}
	$,
	which concludes the proof.
\end{proof}

\begin{proof}[Proof of Lemma~\ref{lem:efficient_score}]
	We first compute the score function for $\beta$:
	\begin{align*}
	S_\beta(Y,X)=&\frac{\partial}{\partial\beta}\log p_{Y,X}\{y,x;\beta,g_0(\cdot),\eta_0(\cdot)\} \mid_{\beta_0}
	\\
	=&\left.\frac{\partial}{\partial\beta}\log\left\{\frac{ \exp(\beta^T XY) g_0(Y)}{\int  \exp(\beta^T XY) dG_0(Y)}\eta_0(X)\right\}\right|_{\beta=\beta_0}
	\\
	=&\left.\frac{\partial}{\partial\beta}\left[\beta^T XY+\log g_0(Y)-\log\left\{\int  \exp(\beta^T XY) dG_0(Y)\right\}+\log\{\eta_0(X)\}\right]\right|_{\beta=\beta_0}
	\\
	=&XY -\left.\frac{\partial}{\partial\beta}\left\{\frac{\int  XY\exp(\beta^T XY) dG_0(Y)}{\int  \exp(\beta^T XY) dG_0(Y)}\right\}\right|_{\beta=\beta_0}
	\\
	=&XY-E(XY|X)\,.
	\end{align*}
	The result follows from Corollary~\ref{cor:proj_on_lambda_bot}.
\end{proof}

\begin{table}[]
	\caption{Simulation results with $\beta_0=(-1,-1)^T$: empirical mean (empirical standard deviation)\label{table_a} }
	{\begin{tabular}{llcccc}
			&	\multicolumn{5}{c}{Setting 1}  \\
			&            & 100          & 200          & 400          & 800          \\
			$\beta_1$    & Lou \& Tsai & -1.06 (0.77) & -1.00 (0.55)    & -1 (0.37)    & -1.01 (0.25) \\
			& Chan        & -1.07 (0.78) & -1.01 (0.56) & -1.00 (0.38)    & -1.01 (0.25) \\
			& Proposed    & -1.06 (0.77) & -1.00 (0.55)    & -1.00 (0.37)    & -1.01 (0.25) \\
			$\beta_2$    & Lou \& Tsai & -1.05 (0.71) & -1.03 (0.51) & -1.01 (0.33) & -1.01 (0.22) \\
			& Chan        & -1.07 (0.72) & -1.03 (0.53) & -1.02 (0.34) & -1.01 (0.23) \\
			& Proposed    & -1.05 (0.71) & -1.03 (0.51) & -1.01 (0.33) & -1.01 (0.22) \\
			Distance & Lou \& Tsai & 0.07 (0.04) & 0.05 (0.03) & 0.03 (0.02) & 0.02 (0.01) \\
			& Chan        & 0.07 (0.04) & 0.05 (0.03) & 0.03 (0.02) & 0.02 (0.01) \\
			& Proposed    & 0.07 (0.04) & 0.05 (0.03) & 0.03 (0.02) & 0.02 (0.01) \\
			&	\multicolumn{5}{c}{Setting 2}  \\
			&             & 100          & 200          & 400          & 800          \\
			$\beta_1$    & Lou \& Tsai & -1.06 (0.24) & -1.03 (0.16) & -1.02 (0.11) & -1.01 (0.08) \\
			& Chan        & -1.05 (0.26) & -1.02 (0.17) & -1.01 (0.12) & -1.01 (0.08) \\
			& Proposed    & -1.06 (0.24) & -1.03 (0.16) & -1.02 (0.11) & -1.01 (0.08) \\
			$\beta_2$     & Lou \& Tsai & -1.07 (0.23) & -1.04 (0.16) & -1.02 (0.10)  & -1.01 (0.08) \\
			& Chan        & -1.05 (0.25) & -1.03 (0.16) & -1.01 (0.11) & -1.01 (0.08) \\
			& Proposed    & -1.07 (0.23) & -1.04 (0.16) & -1.02 (0.10)  & -1.01 (0.08) \\
			Distance & Lou \& Tsai & 0.39 (0.20)  & 0.28 (0.14) & 0.20 (0.10)  & 0.13 (0.07) \\
			& Chan        & 0.40 (0.21) & 0.29 (0.14) & 0.20 (0.10)  & 0.14 (0.07) \\
			& Proposed    & 0.39 (0.20)  & 0.28 (0.14) & 0.20 (0.10)  & 0.13 (0.07)
	\end{tabular}}
\end{table}

\begin{figure}[!t]
	\begin{center}
		\includegraphics[width=1\linewidth]{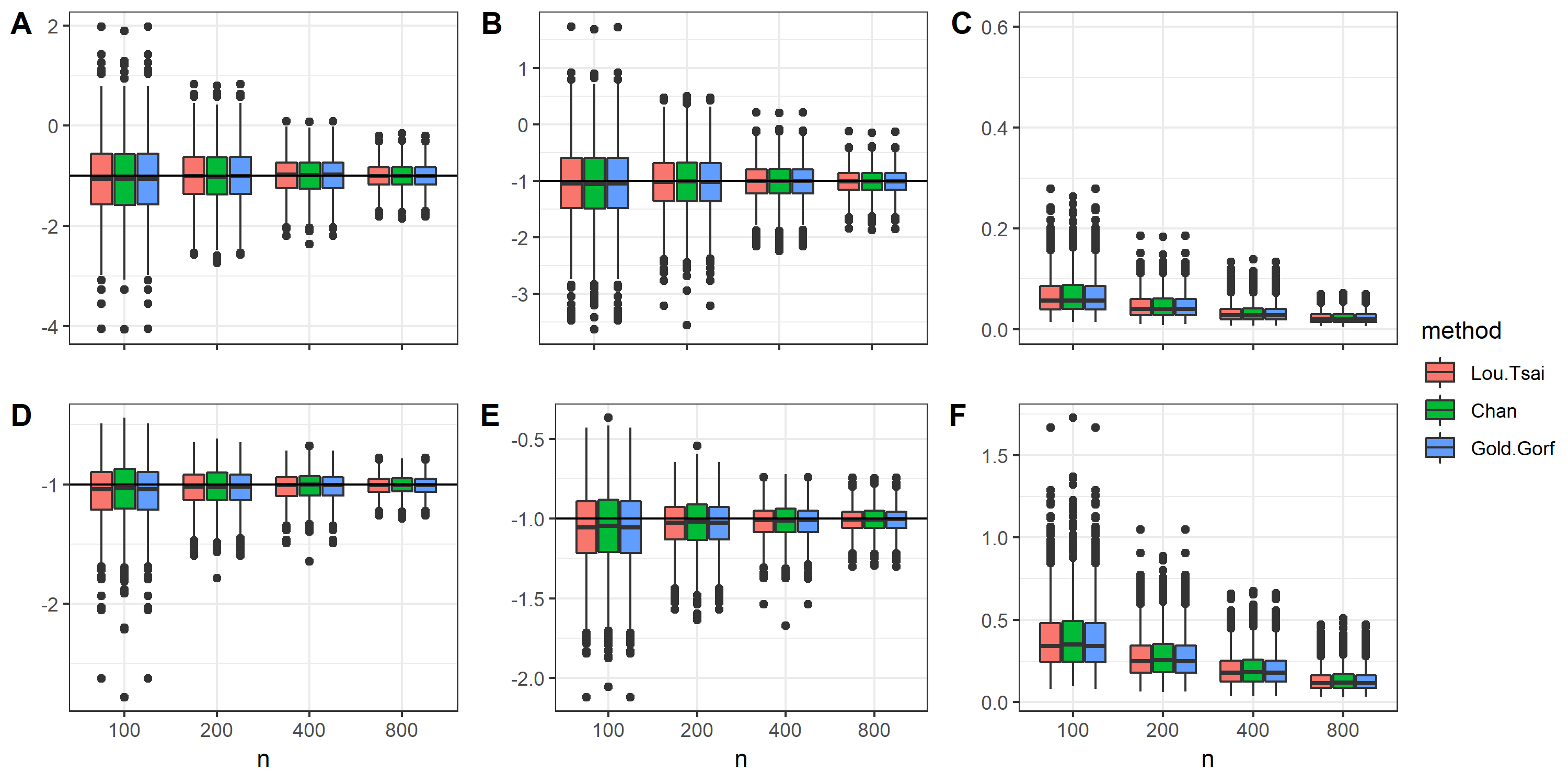}
		\caption{Results of Settings~1 and~2 are summarized in the top and the bottom lines, respectively. A~and~D present the bias related to $\beta_1$, B~and~E present that of $\beta_2$, and~C and~F  present the bias related to $G$. 
			\label{fig_12}}
	\end{center}
\end{figure}


\bibliographystyle{plainnat}
\bibliography{PLR}

\end{document}